\documentclass[12pt]{amsart}

\usepackage{latexsym,amsfonts,amssymb,epsfig,verbatim}
\usepackage{amsmath,amsthm,amssymb,latexsym,graphics,textcomp}
\usepackage{eufrak}
\usepackage{graphicx,pinlabel}
\input{xy}
\xyoption{all}
\usepackage[mathscr]{eucal}

\DeclareMathAlphabet{\mathcal}{OMS}{cmsy}{m}{n}

\newcommand{\bc}{\mathbb C}
\newcommand{\bd}{\mathbb D}

\newcommand{\bh}{\mathbb H}

\newcommand{\bz}{\mathbb Z}

\newcommand{\br}{\mathbb R}

\newcommand{\cM}{\mathcal M}

\newcommand{\sx}{\mathscr X}
\newcommand{\sy}{\mathscr Y}

\newcommand{\cI}{\mathcal I}
\newcommand{\sF}{\mathcal F}

\newcommand{\la}{\langle}
\newcommand{\ra}{\rangle}

\newcommand{\hra}{\hookrightarrow}

\newcommand{\wt}{\widetilde}
\newcommand{\ti}{\tilde}

\newcommand{\al}{\alpha}
\newcommand{\be}{\beta}

\newcommand{\vp}{\varphi}
\newcommand{\ep}{\epsilon}

\DeclareMathOperator{\Aut}{Aut}

\DeclareMathOperator{\Mod}{Mod}
\DeclareMathOperator{\Teich}{Teich}
\DeclareMathOperator{\arccosh}{arccosh}

\DeclareMathOperator{\teich}{Teich}

\DeclareMathOperator{\QC}{QC}

\newtheorem{Thm}{Theorem}[section]
\newtheorem{Prop}[Thm]{Proposition}
\newtheorem{Lem}[Thm]{Lemma}

\newtheorem{Question}[Thm]{Question}

\theoremstyle{definition}

\theoremstyle{remark}
\newtheorem{Rem}[Thm]{Remark}

\numberwithin{equation}{section}

\begin{document}
\title[Quasiconformal homogeneity of hyperbolic surfaces]{Quasiconformal homogeneity and subgroups of the mapping class group}
\author{Nicholas G. Vlamis}
\address{Department of Mathematics \\ Boston College \\ 140 Commonwealth Ave. \\ Chestnut Hill, MA 02467 \\ nicholas.vlamis@bc.edu}

%Abstract

\begin{abstract}
In the vein of Bonfert-Taylor, Bridgeman, Canary, and Taylor we introduce the notion of quasiconformal homogeneity  for closed oriented hyperbolic surfaces restricted to subgroups of the mapping class group. We find uniform lower bounds for the associated quasiconformal homogeneity constants across all closed hyperbolic surfaces in several cases, including the Torelli group, congruence subgroups, and pure cyclic subgroups.  Further, we introduce a counting argument providing a possible path to exploring a uniform lower bound for the nonrestricted quasiconformal homogeneity constant across all  closed hyperbolic surfaces.
\end{abstract}

\maketitle

%Introduction

\section{Introduction}

Let $M$ be a hyperbolic manifold and $\QC(M)$ be the associated group of quasiconformal homeomorphisms from $M$ to itself. Given any subgroup $\Gamma\leq \QC(M)$, we say that $M$ is {\it $\Gamma$-homogeneous }if the action of $\Gamma$ on $M$ is transitive.  Furthermore, we say $M$ is {\it $\Gamma_K$-homogeneous} for $K\in[1,\infty)$ if the restriction of the action of $\Gamma$ on $M$ to the subset
$$\Gamma_K = \{f\in \Gamma\colon K_f\leq K\}$$
on $M$ is transitive, where $K_f = \inf\{ K \colon f \text{ is } K\text{-quasiconformal}\}$ is the \textit{dilatation} of $f$.  

If $\Gamma = \QC(M)$ and there exists a $K$ such that $M$ is $\Gamma_K$-homogeneous, then this manifold is said to be {\it $K$-quasiconformally homogeneous}, or {\it $K$-qch}.  In \cite{bcmt} it is shown that for each $n\geq 3$ there exists a constant $K_n>1$ such that if $M\neq \bh^n$ is an $n$-dimensional $K$-quasiconformally homogeneous hyperbolic manifold, then $K\geq K_n$.  This result relies on rigidity in higher dimensions, which does not occur in dimension two.  The natural question motivating this paper is as follows:

\begin{Question}\label{q:k2}
Does there exist a constant $K_2>1$ such that every $K$-qch surface $X\neq \bh^2$ satisfies $K\geq K_2$?
\end{Question}

Let $\mathrm{Homeo}^+(S)$ be the group of orientation preserving homeomorphisms of a surface $S$, then the {\it mapping class group} of $S$ is defined to be $\pi_0(\mathrm{Homeo}^+(S))$ and is denoted $\Mod(S)$.  Given a closed hyperbolic surface and $f\in \QC(X)$, let $[f]\in \Mod(X)$ denote its homotopy class, which gives a surjection $\pi: \QC(X)\to \Mod(X)$, where $f\mapsto [f]$.  If $H\leq \Mod(X)$, we say that $X$ is {\it $H$-homogeneous} if $X$ is $\pi^{-1}(H)$-homogeneous. Similarly, we say $X$ is {\it $H_K$-homogeneous} if it is $\pi^{-1}(H)_K$-homogeneous.

The focus of this paper will be to restrict ourselves to homogeneity with respect to subgroups of the mapping class group of closed hyperbolic surfaces and find lower bounds for the associated homogeneity constants.  We will go about this by leveraging lower bounds on the quasiconformal dilatations for maps in a given homotopy class.

\bigskip

\noindent{\bf Torelli and Congruence Subgroups.}
Let $S$ be a closed orientable surface, then $\Mod(S)$ acts on the first homology $H_1(S,\bz)$ by isomorphisms and the kernel of this action is called the {\it Torelli group}, denoted $\cI(S)$. Similarly, the kernel of the action of $\Mod(S)$ on $H_1(S, \bz/r\bz)$ is called the {\it level $r$ congruence subgroup} and is denoted by $\Mod(S)[r]$.  The first theorem gives a universal bound on the quasiconformal homogeneity constant with respect to these subgroups for closed hyperbolic surfaces.

\begin{Thm}\label{torelli}
There exists a constant $K_T>1$ such that if $X$ is a closed hyperbolic surface that is $\Gamma_K$-homogeneous for $\Gamma=\cI(X)$ or $\Gamma=\Mod(X)[r]$ with $r\geq 3$, then $K\geq K_T$.
\end{Thm}

\noindent
The case of $\Gamma = \cI(X)$ was independently discovered by Greenfield \cite{greenfield}.

Since $H_1(S, \bz/r\bz)$ is a finite group, so is its automorphism group; hence, $\Mod(S)[r]$ is finite index in Mod(S). Theorem \ref{torelli}  provides an optimistic outlook for answering Question \ref{q:k2} in the positive for the case of closed surfaces.

\bigskip

\noindent{\bf Homogeneity and Teichm\"uller Space.}
The rest of the paper is flavored by a technique, introduced in Section \ref{section:orbit}, which translates questions about homogeneity constants to questions about orbit points under the action of the mapping class group on Teichm\"uller space.  Given a closed hyperbolic surface $S$, we define its associated {\it Teichm\"uller space} $\teich(S)$ to be the space of equivalence classes of pairs $(X,\vp)$, where $X$ is a hyperbolic surface and $\vp:S\to X$ is a homeomorphism called the {\it marking}.  Two such pairs $(X,\vp)$ and $(Y,\psi)$ are equivalent if $\psi\circ \vp^{-1}:X\to Y$ is homotopic to an isometry  (see \cite{hubbard}). The mapping class group $\Mod(S)$ acts on $\teich(S)$ by changing the marking: 
$$[f]\cdot [(X,\vp)] = [(X, \vp\circ f^{-1})].$$
Furthermore, this action is by isometries with respect to the Teichm\"uller metric on $\teich(S)$, which is defined by
$$d_T([(X,\vp)],[(Y,\psi)]) = \frac12 \log(\min K(h)),$$
where the minimum of the quasiconformal dilatation is over all quasiconformal maps $h: X\to Y$ homotopic to $\psi\circ\vp^{-1}$.  The fact that this minimum exists is a well-known theorem of Teichm\"uller (a proof can be found in \cite{hubbard}).  

Our next theorem is a direct result of the technique mentioned above and gives a possible path to finding a lower bound for the quasiconformal homogeneity constant for closed hyperbolic surfaces.   It is shown in \cite{bcmt} (see Proposition \ref{diameter} below) that surfaces with short curves have large homogeneity constants.  
We let 
$$\teich_{(\ep,\infty)}(S) = \{[(X,\vp)]\in\teich(S)\colon \ell(X)>\ep\},$$
where $\ell(X)$ is the length of the systole.
Also, given a point $\sx \in \teich(S)$, let $B_R(\sx)$ be the ball of radius $R$ about $\sx$ in $(\teich(S),d_T)$.  We let $S_g$ be an oriented closed genus $g$ surface.

\begin{Thm}\label{thm:counting}
Suppose there exist constants $\ep, R, C > 0$ such that for any  $\sx\in \teich_{(\ep,\infty)}(S_g)$ with $g>1$
$$|\{f \in \Mod(S_g) \colon f\cdot \sx \in B_{R}(\sx)\}| \leq Cg.$$
Then, there exists a constant $K_2>1$ such that any closed $K$-qch surface must have $K\geq K_2$. 
\end{Thm}

\begin{Question}
Does there exist such an $\ep, R, C$?
\end{Question} 
\noindent
Note that $\ep$ and $C$ can be chosen to be arbitrarily large and $R$ can be chosen to be arbitrarily small.

\bigskip

\subsection*{ Finite, Cyclic, and Torsion-Free Subgroups.}  Returning to more restrictive forms of homogeneity, we use this counting method to consider finite and cyclic subgroups of the mapping class group:

\begin{Thm}\label{finite}
There exists a constant $K_F>1$ such that if a closed hyperbolic surface $X$ is $\Gamma_K$-homogeneous, where $\Gamma<\Mod(X)$ has finite order, then $K\geq K_F$. Furthermore, we have
$$K_F \geq \sqrt{\psi\left(2\arccosh\left(\frac1{42}+1\right)\right)} = 1.11469\ldots,$$
where $\psi$ is defined in equation \eqref{eq:psi1}.
\end{Thm}

\begin{Thm}\label{cyclic}
There exists a constant $K_C>1$ such that if a closed hyperbolic surface $X$ is $\Gamma_K$-homogeneous, where $\Gamma = \la[f]\ra$ with $[f]\in \Mod(X)$ a pure mapping class, then $K\geq K_C$. Furthermore, we have
$K_C \geq 1.09297.$
\end{Thm}

It is particularly difficult to understand the orbit of points in $\teich(S)$ under periodic mapping classes; hence, our last theorem deals with torsion-free subgroups of $\Mod(S)$.

\begin{Thm}\label{thm:torsion}
Let $X$ be a closed hyperbolic surface and suppose $\Gamma<\Mod(X)$ is torsion-free.  If $X$ is $\Gamma_K$-homogeneous, then 
$$\log K \geq \frac{1}{7000g^2},$$
where $g$ is the genus of $X$.
\end{Thm}

\begin{Question}
Can one find a constant $C$ such that every closed $K$-qch surface satisfies $K\geq Cg^{-2}$?
\end{Question}

The rest of the paper discusses how to define continuous functions on Teichm\"uller space and Moduli space using subgroups of the mapping class group and the associated homogeneity constants for surfaces.  

\subsection*{Related Results in the Literature.}

In recent years there have been several papers published that make progress towards understanding quasiconformal homogeneity of surfaces.  In \cite{bbc} the authors bound  the quasiconformal constant of hyperbolic surfaces having automorphisms with many fixed points away from 1, in particular, all hyperelliptic surfaces. In the same paper, they also consider homogeneity with respect to $\Gamma = \{e\} < \Mod(X)$ and $\Aut(X)$.  They prove that a surface is $\{e\}_K$-homogeneous for some $K$ if and only if it is closed; furthermore, there exists a constant $K_e>1$ such that $K\geq K_e$.   In a similar fashion, the authors find that a hyperbolic surface $X$ is $\Aut(X)_K$-homogeneous for some $K$ if and only if it is a regular cover of a hyperbolic orbifold; furthermore, there exists a constant $K_{aut}>1$ such that $K\geq K_{aut}$. A sharp bound is found for the constant $K_{aut}$ in \cite{btgr}. The authors in \cite{markovic} show the existence of a lower bound $K_0>1$ for the quasiconformal homogeneity constant of genus zero surfaces, which answers a question about quasiconformal homogeneity of planar domains posed by Gehring and Palka in \cite{gp}.

\subsection*{Acknowledgements}
I would like to thank my adviser, Martin Bridgeman, for his guidance. I would also like to thank Ian Biringer, especially for his help with the geometric convergence argument in Lemma \ref{lem:partial}, Andrew Yarmola for helpful conversations, and the reviewer for pointing out improvements to the paper.

% Background

\section{Background}

\subsection{Quasiconformal Geometry}
We may think of a quasiconformal map $f:\bc\to\bc$ as a function whose derivative $df_p$ sends the unit circle in $T_p\bc$ to an ellipse in $T_{f(p)}\bc$ whose ratio of the major to minor axis we call $K_f(p)$, wherever the derivative is defined, and $K_f(p)$ is required to be bounded uniformly for all $p\in \bc$. We let $K_f$ or $K(f)$ denote the \textit{dilatation} of $f$, which is defined to be the supremum of $K_f(p)$ over all of $\bc$. As this is a local condition, this notion holds for Riemann surfaces.  As $K_{f\circ g} \leq K_f\cdot K_g$, we see that 
$$\QC(X)=\{f:X\to X\, |\,  f \text{  is a quasiconformal homeomorphism}\}$$
is a group.  We refer the reader to \cite{hubbard} and \cite{fm} for details.

There are two  properties of quasiconformal maps that will play a key role in what follows.  The first property shows us that quasiconformal maps retain some of the nicety of conformal maps.  Let $\bd$ denote the unit disk in $\bc$.  The following theorem and proof can be found in \cite{hubbard}.

\begin{Thm}\label{thm:normal}
Denote by $\sF_K(\bd)$ the set of $K$-quasiconformal homeomorphisms $f:\bd\to\bd$ with $f(0)=0$, then $\sF_K(\bd)$ is a normal family.
\end{Thm}
\noindent
We will rely heavily on this theorem for the convergence of sequences of quasiconformal maps, especially in understanding the continuity of the quasiconformal homogeneity constants on moduli space. 

The next property relates the quasiconformal condition of a homeomorphism $f:\bd\to\bd$ to the geometry of the hyperbolic plane.  We say that $f:\bd\to\bd$ is an $(A,B)$-quasi-isometry if there are constants $A,B>0$ such that 
$$\frac{d(z,w)}{A}-B\leq d(f(z),f(w))\leq Ad(z,w)+B,$$
for all $z,w\in\bd$ and where $d$ is the hyperbolic metric on $\bd$.  The following theorem can be found in \cite{vuorinen}.

\begin{Thm}\label{thm:qi}
Let $f:\bd\to\bd$ be $K$-quasiconformal, then $f$ is a $(K,K\log 4)$-quasi-isometry with respect to the hyperbolic metric.
\end{Thm}
In particular, the image of a geodesic $\gamma\in \bd$ under a $K$-quasiconformal $f:\bd\to\bd$ is a $(K, K\log4)$-quasi-geodesic.  It is well known (see \cite{kapovich}) that a quasi-geodesic stays within a bounded distance of a geodesic.  In our case, we know there exists some $C(K)$ and some geodesic $\tilde\gamma$ such that $f(\gamma)\subset N_{C(K)}(\tilde\gamma)$, where $N_{C(K)}$ is the $C(K)$-neighborhood.

\subsection{Quasiconformal Homogeneity}

The main goal of this section is to state one of the main results of \cite{bcmt}, which describes how quasiconformal homogeneity interacts with the geometry of a manifold.  Though this paper is focused on surfaces, their work deals with arbitrary dimension, so for the moment we will work in the general setting of hyperbolic manifolds.  If $M$ is an orientable hyperbolic $n$-manifold then there exists a discrete subgroup $\Gamma< \mathrm{Isom}^+(\bh^n)$, called a {\it Kleinian group}, so that $M$ is isometric to $\bh^n/\Gamma$.  The action of $\Gamma$ extends to $\partial\bh^n = \mathbb{S}^{n-1}$ and acts by conformal automorphisms.  The {\it limit set} $\Lambda(\Gamma)$ is defined to be the intersection of the closure of an orbit of a point $x\in \bh^n$ with $\partial\bh^n$, that is $\Lambda(\Gamma) = \overline{\Gamma\cdot x} \cap \partial\bh^n$ (note this definition is independent of the choice of $x$). See \cite{thurston} for more on Kleinian groups.  Also define $\ell(M)$ to be the infimum of the lengths of homotopically non-trivial curves in $M$ and define $d(M)$ to be the suprememum of the diameters of embedded hyperbolic balls in $M$.

\begin{Thm}[Theorem 1.1 in \cite{bcmt}]\label{diameter}
For each dimension $n\geq 2$ and each $K\geq 1$, there is a positive constant $m(n,K)$ with the following property.  Let $M = \bh^n/\Gamma$ be a $K$-quasiconformally homogeneous hyperbolic $n$-manifold, which is not $\bh^n$.  Then 
\begin{enumerate}
\item[(1)] $d(M) \leq K\ell(M)+2K\log4$.
\item[(2)] $\ell(M)\geq m(n,K)$, i.e. there is a lower bound on the injectivity radius of $M$ that only depends on $n$ and $K$.
\item[(3)] Every nontrivial element of $\Gamma$ is hyperbolic and the limit set $\Lambda(\Gamma)$ of $\Gamma$ is $\partial\bh^n$.
\end{enumerate}
\end{Thm}

In addition,  every closed manifold is $K$-quasiconformally homogeneous for some $K$ (also in \cite{bcmt}).
These facts tell us that a geometrically-finite hyperbolic surface is $K$-quasiconformally homogeneous for some $K$ if and only if it is closed.  Observe that if $G<G'<\Mod(X)$ for some hyperbolic surface $X$, then if $X$ is $G_K$-homogeneous we have that $X$ is also $G'_K$-homogeneous.  In particular, a geometrically-finite hyperbolic surface $X$ is $G_K$-homogeneous for $G<\Mod(X)$ if and only if $X$ is closed.  This fact will be our motivation for stating our theorems in terms of closed surfaces as opposed to the geometrically-finite terminology.

The other key tool we will need comes from understanding the quasiconformal homogeneity constant under geometric convergence and the fact that the only hyperbolic $n$-manifold that is 1-quasiconformally homogeneous is $\bh^n$.  

\begin{Prop}[Proposition 3.2 in \cite{bbc}]\label{convergence}
Let $\{M_i\}$ be a sequence of hyperbolic manifolds with $M_i$ being $K_i$-quasiconformally homogeneous.  If $\displaystyle{\lim_{i\to\infty} K_i = 1}$, then $\displaystyle{\lim_{i\to\infty} \ell(M_i) = \infty}$.
\end{Prop}
 
%The Torelli Case

\section{Torelli Groups and Congruence Subgroups}

For a closed orientable surface $S_g$ with genus $g\geq2$, the {\it Torelli group}, $\mathcal{I}(S_g)$,  is the kernel of the action of $\Mod(S_g)$ on $H_1(S_g, \bz)$, the first homology with $\bz$ coefficients.  We similarly define the {\it level $m$ congruence subgroup}, $\Mod(S_g)[m]$, as the kernel of the action of $\Mod(S_g)$ on $H_1(S_g, \bz/m\bz)$. For the rest of this section all the results stated will hold for both classes of subgroups just mentioned with $m\geq 3$ in the latter case;  we will set $\Gamma(S) = \cI(S),\Mod(S)[m]$.

An element $f\in\Mod(S)$ is called {\it pseudo-Anosov} if it has infinite order and no power of $f$ fixes the isotopy class of any essential 1-submanifold.  
Let $\teich(S)$ denote the Teichm\"uller space, the parameterization space of hyperbolic structures, associated to $S$.
Given any $f\in \Mod(S)$ define 
\begin{equation}\label{eq:tau}
\tau(f) = \inf_{\sx\in\teich(S)}\{d_T(\sx, f\cdot\sx)\},
\end{equation}
then $f$ is pseudo-Anosov if and only if $\tau(f)>0$ and is realized by some $\sx\in\teich(S)$ (see \cite{bers}).  If $f$ is pseudo-Anosov, we define its \textit{dilatation} to be $\lambda(f) = \exp(\tau(f)).$

In \cite{flm} the authors prove that for a pseudo-Anosov element $f\in \Gamma(S)$ that $\log \lambda(f) \geq 0.197$.  We would like to have a similar result for reducible elements of these subgroups.  We can get such a result directly from the authors' original proof with understanding how their pseudo-Anosov assumption is being used.

In their proof, they use a cone metric on $S$ coming from a quadratic differential with stable and unstable foliations corresponding to the stable and unstable foliations for $f$.  They use this metric to compare lengths of curves.  The same proof can be given using a hyperbolic metric on $S$ yielding $2\tau(f)=\log(\lambda(f)^2)\geq 0.197$.  The authors' proof over a hyperbolic metric views $f$ as a quasiconformal map and uses Wolpert's lemma: 

\begin{Lem}[Wolpert's Lemma, Lemma 12.5 in \cite{fm}]
Let $X, Y$ be hyperbolic surfaces and let $f: X\to Y$ be a $K$-quasiconformal homeomorphism.  For any isotopy class $c$ of simple closed curves in $X$, the following holds:
$$\frac{\ell_X(c)}{K}\leq \ell_Y(f(c))\leq K\ell_X(c),$$
where $\ell_X(c)$ denotes the length of the unique geodesic representative of $c$ in $X$.
\end{Lem}
\noindent
This is also explained in a remark in \cite{flm}.  By replacing the cone metric coming from the pseudo-Anosov with a hyperbolic metric we remove the first instance of the pseudo-Anosov assumption. 

The second way that the pseudo-Anosov assumption is used is to state that $f$ does not fix the homotopy class of a shortest curve.  We can remove this assumption by looking at mapping classes that do not fix a shortest curve:

\begin{Thm}[Farb, Leininger, Margalit, \cite{flm}]
Let $X$ be a hyperbolic surface and $\gamma$ the homotopy class of a shortest curve in $X$. If $f:X\to X$ is a quasiconformal homeomorphism with $[f]\in \cI(X)$ or $[f]\in \Mod(X)[m]$ for some $m\geq 3$ such that $f(\gamma)\neq \gamma$, then $\log K(f) \geq 0.197$.
\end{Thm}

For studying quasiconformal homogeneity with respect to $\Gamma(S)$, this theorem will allow us to discard any elements not fixing a shortest curve.  This will be enough to prove our theorem.  We start with a lemma describing the situation for large genus surfaces.

\begin{Lem}\label{lem:kt}
There exists $g_0$ such that if $X$ is a closed hyperbolic surface of genus $g>g_0$ and $X$ is $\Gamma_K$-homogeneous for either  $\Gamma=\cI(X)$ or $\Gamma=\Mod(X)[m]$ for $m\geq 3$, then $\log K >0.197$.  
\end{Lem}

\begin{proof}
From Theorem \ref{thm:qi} above, we know that if $f:X\to X$ is $K$-quasiconformal, then $f$ is a $(K, K\log 4)$-quasi-isometry.  In particular, there is some $C(K)\geq0$ such that if $\gamma$ is a geodesic in $X$, then $f(\gamma)$ is contained in a $C(K)$-neighborhood of $\tilde\gamma$, call it $N_{C(K)}(\tilde\gamma)$,  for some geodesic $\tilde\gamma$ in $X$.  Define $C_0 = C(\exp(0.197))$.  Also, if $X$ is a genus $g$ hyperbolic surface, then $\ell(X) \leq A\log g$, where $A$ is a constant independent of genus (this is Gromov's inequality for surfaces, see \cite{gromov}).  Now choose $g_0$ such that 
$$\frac{4\pi(g_0-1)}{A\log g_0}>2\sinh C_0.$$

Assume that the genus of $X$ is $g>g_0$ and that $X$ is $\Gamma_K$-homogeneous.  Let $\gamma$ be a closed geodesic in $X$ of shortest length, then it satisfies $\ell_X(\gamma) \leq A\log g$. For every $y\in X$ and $x\in \gamma$ there exists $f:X\to X$ such that $[f]\in\Gamma_K$ and $f(x)=y$.  If $\log K <0.197$, then $[f(\gamma)]=[\gamma]$ implying every point of $X$ must be in the $C_0$-neighborhood of $\gamma$.  Let us identify the universal cover of $X$ with $\bh^2$, so that $X = \bh^2/G$ for $G<\mathrm{Isom}^+(\bh^2)$.  In the upper half plane model we can translate a lift of $\gamma$ to be the imaginary axis so that the geodesic segment $[i,ie^{\ell_{X}(\gamma)}]$ maps onto $\gamma$.  If $U$ is a $C_0$-neighborhood of this segment in $\bh^2$, then from above we know there exists a fundamental domain for the action of $G$ on $\bh^2$ contained in $U$. In particular, this implies $\mathrm{Area}(U)\geq \mathrm{Area}(X)$. However, 
$$\mathrm{Area}(U) = 2\ell_{X}(\gamma) \sinh C_0 < 2A\log(g)\sinh(C_0)< 4\pi(g-1).$$
But, $4\pi(g-1) = \mathrm{Area(X)}$; hence, we found $\mathrm{Area}(U)<\mathrm{Area}(X)$.  This is a contradiction; thus, we must have $\log K >0.197$. 
\end{proof}

\noindent
{\bf Theorem \ref{torelli}} 
{\it There exists a constant $K_T>1$ such that if $X$ is a closed hyperbolic surface that is $\Gamma_K$-homogeneous for $\Gamma=\cI(X)$ or $\Gamma=\Mod(X)[r]$ with $r\geq 3$, then $K\geq K_T$.}

\begin{proof}
Given a sequence of hyperbolic surfaces $\{X_n\}$, let $g_n$ be the genus of $X_n$ and $\Gamma_n = \cI(X_n),\Mod(X_n)[m]$  for $m\geq 3$.  We proceed by contradiction:  Suppose the statement is false, then there exists a sequence of hyperbolic surfaces $\{X_n\}$  that are $(\Gamma_n)_{K_n}$-homogeneous such that $\displaystyle{\lim_{n\to\infty} K_n = 1}$. As $K_n \to 1$, Proposition \ref{convergence} tells us that $\ell(X_n) \to \infty$ and Gromov's inequality implies that  $g_n\to\infty$.  Pick $N$ such that $g_N>g_0$, where $g_0$ is from Lemma \ref{lem:kt}.  For all $n>N$ we have $\log K_n > 0.197$ contradicting $K_n\to 1$.  This completes the proof.
\end{proof}

%Counting orbit points in Teichmuller Space

\section{A Counting Problem in Teichm\"uller Space}\label{section:orbit}

For the rest of the paper, our main method of studying quasiconformal homogeneity will be to translate the problem of understanding the homogeneity constants to one of counting orbit points in Teichm\"uller space under the action of the mapping class group. Before stating the lemma that will allow us to accomplish this we recall a proposition in \cite{bbc}:

\begin{Prop}[Proposition 6.2 in \cite{bbc}]\label{psi}
Let $f:\bh^2\to\bh^2$ be a quasiconformal map which extends to the identity on $\partial_\infty\bh^2$ and let $x\in \bh^2$.  Then $K(f)\geq \psi(d(x,f(x)))$, where $\psi: [0,\infty)\to [1,\infty)$ is the increasing homeomorphism given by the function
\begin{equation}\label{eq:psi1}
\psi(d) = \coth^2\left(\frac{\pi^2}{4\mu(e^{-d})}\right)=\coth^2\mu\left(\sqrt{1-e^{-2d}}\right),
\end{equation}
where $\mu(r)$ is the modulus of the Gr\"otsch ring whose complementary components are $\overline{\mathbb{B}^2}$ and $[1/r,\infty]$ for $0<r<1$.
\end{Prop}
\noindent
The explicit formula for $\psi$ was originally due to Teichm\"uller \cite{teich}.
In what follows, we will define $K(\vp)$ for $\vp\in \Mod(X)$ by
$$K(\vp) = \min\{K_f \colon f\in \QC(X) \text{ and } [f] = \vp\},$$
where $[f]$ denotes the homotopy class of $f$.

\begin{Lem}\label{lemma}
Let $X$ be a genus $g$ closed hyperbolic surface and $\Gamma<\Mod(X)$ such that $X$ is $\Gamma_K$-homogeneous. If the set $$\{\vp \in \Gamma \colon K(\vp)< K\}$$ is finite with cardinality $n$, then $$K\geq \sqrt{\psi\left(2\arccosh\left(\frac 2n(g-1)+1\right)\right)},$$
where $\psi$ is defined in \eqref{eq:psi1}.
\end{Lem}

\begin{proof}
As the action of $\Mod(X)$ on $\teich(X)$ is properly discontinuous there can only be finitely many mapping classes with dilatation less than $K$.
Let $\vp_1, \ldots, \vp_n$ be the $n$ elements in $\Gamma$ such that $K(\vp_i) \leq K$.  Fix $a\in X$ and let 
$$U_i = \{ x\in X\colon \exists f\in QC_K(X)  \text{ such that } [f]=\vp_i \text{ and } f(a) = x\}.$$
 In particular, $X = \bigcup_{i=1}^n U_i$.  Now Area$(X) = 4\pi(g-1) \leq \sum \text{Area}(U_i)$; hence, there exists $k\in\{1,\ldots, n\}$ such that $U = U_k$ satisfies Area$(U)\geq \frac{4\pi}n (g-1)$.  
Let $d$ be the diameter of $U$ so that 
$$2\pi\left(\cosh\frac d2  - 1\right) \geq \text{Area}(U) \geq \frac{4\pi}n (g-1),$$
where the leftmost term is the area of the hyperbolic ball of diameter $d$. 
This implies 
$$d\geq 2\arccosh\left(\frac2n(g-1) +1\right).$$  
For $\ep>0$, let $x,y\in U$ such that $d_X(x,y) = d-\ep$ and pick $f,g \in\QC_K(X)$ with $[f]=[g]=\vp_i$ such that $f(a) = x$ and $g(a) = y$, then $h=g\circ f^{-1}$ is isotopic to the identity and $h(x) = y$.  Let $\tilde h:\bh^2\to\bh^2$ be a lift of $h$ which extends to the identity on $\partial_\infty \bh^2$.  The above proposition implies $$K(\tilde h) = K(h) \geq \psi(d(x,y))=\psi(d-\ep).$$  We now have
$$K^2\geq K(f)\cdot K(g^{-1}) \geq K(f\circ g^{-1}) =K(h) \geq \psi(d_X(x,y)) = \psi(d-\ep).$$
The result follows by letting $\epsilon$ tend to zero and the fact that $\psi$ is increasing.
\end{proof}

Let us wrap the above lemma in the language of Teichm\"uller theory.  Given $\sx = (X, \vp)\in \teich(S_g)$ we can identify $f\in \Mod(S_g)$ with $\vp\circ f\circ \vp^{-1}\in \Mod(X)$,  then
$$|\{g \in \Mod(X) \colon K(g)< K\}| = |\{f \in \Mod(S_g) \colon f\cdot \sx \in B_{\log\sqrt K}(\sx)\}|, $$
where $B_R(\sx)$ is the ball of radius $R$ in the Teicm\"uller metric centered at $\sx \in \teich(S_g).$ 
This allows us to think about orbits in $\teich(S_g)$.  
Lemma \ref{lemma} provides a possible route to proving that there exists an universal constant $K_2>1$ such that if $X$ is a $K$-quasiconformally homogeneous closed hyperbolic  surface, then $K\geq K_2$.

\bigskip

\noindent
{\bf Theorem \ref{thm:counting}} {\it
Suppose there exist constants $\ep, R, C > 0$ such that for any  $\sx\in \teich_{(\ep,\infty)}(S_g)$ with $g>1$
$$|\{f \in \Mod(S_g) \colon f\cdot \sx \in B_{R}(\sx)\}| \leq Cg.$$
Then, there exists a constant $K_2>1$ such that any closed $K$-qch surface must have $K\geq K_2$. 
}

\begin{proof}
We proceed by contradiction:  Assume there exists a sequence of closed hyperbolic surfaces $\{X_n\}$ such that $X_n$ is $K_n$-quasiconformally homogeneous and $K_n\to 1$.  This implies $\ell(X_n) \to \infty$ by Proposition \ref{convergence} and $g_n\to \infty$ by Gromov's inequality, where $g_n$ is the genus of $X_n$.  
By Lemma \ref{lemma} and the cardinality assumption we have  that
$$K_n\geq \sqrt{\psi\left(2\arccosh\left(\frac 2{Cg_n}(g_n-1)+1\right)\right)}.$$
(Note that we use that both $\psi$ and $\arccosh$ are increasing functions.)
In particular, we have
$$\lim_{n\to\infty} K_n \geq \sqrt{\psi\left(2\arccosh\left(\frac 2C+1\right)\right)}>1.$$
This contradicts the assumption $K_n \to 1$, which completes the proof.
\end{proof}

%Finite Subgroups

\section{Finite Subgroups}

For a closed orientable surface $S$ with negative Euler characteristic, there are well known bounds for the order of finite groups and elements in $\Mod(S)$:  it is a theorem of Hurwitz that the the group $\mathrm{Isom}^+(X)$ for a closed hyperbolic surface $X$ of genus $g\geq 2$ has order bounded above by $84(g-1)$.  Also, it was proved by Wiman \cite{wiman} that any element in $\mathrm{Isom}^+(X)$ has order bounded above by $4g+2$ (both of these are proved in \cite{fm}).   In addition, the Nielsen realization theorem proved by Kerckhoff \cite{k} tells us that a finite subgroup of $\Mod(S)$ can be realized as a subgroup of $\mathrm{Isom}^+(X)$ for some hyperbolic surface $X$ homeomorphic to $S$.
Combining these results with Lemma \ref{lemma}, we get the following results:

\bigskip\noindent
{\bf Theorem \ref{finite}} {\it
There exists a constant $K_F>1$ such that if a closed hyperbolic surface $X$ is $\Gamma_K$-homogeneous, where $\Gamma<\Mod(X)$ has finite order, then $K\geq K_F$. Furthermore, we have
$$K_F \geq \sqrt{\psi\left(2\arccosh\left(\frac1{42}+1\right)\right)} = 1.11469\ldots,$$
where $\psi$ is defined in equation \eqref{eq:psi1}.}

\begin{proof}
From the above discussion, we know that $|\Gamma| \leq 84(g-1)$. The result follows by setting $n= 84(g-1)$ in Lemma \ref{lemma}.  
\end{proof}

%{\bf Theorem \ref{periodic}} {\it
\begin{Thm}
There exists a constant $K_P>1$ such that if a closed hyperbolic surface $X$ is $\Gamma_K$-homogeneous, where $\Gamma=\la f\ra$ and $f\in\Mod(X)$ is periodic, then $K\geq K_P$. In particular, we have 
$$K_P\geq\sqrt{\psi\left(2\arccosh\left(\frac65\right)\right)} = 1.35547\ldots .$$
\end{Thm}

\begin{proof}
From the above discussion, we know that $|\vp| \leq 4g+2$, so we can use Lemma \ref{lemma} with $n = 4g+2$.  We see the worst case is $n=4g+2$ when $g=2$.
\end{proof}

%Pure Cyclic Subgroups

\section{Pure Cyclic Subgroups}

We follow \cite{ivanov} in calling a homeomorphism $f: S\to S$ {\it pure} if for some closed one-dimensional submanifold $C$ of $S$ the following are true:
\begin{itemize}
\item[(1)] the components of $C$ are nontrivial, 
\item[(2)] $f|_C$ is the identity, 
\item[(3)] $f$ does not rearrange the components of $S\smallsetminus C$, and 
\item[(4)] $f$ induces on each component of $S$ cut along $C$ a homeomorphism either homotopic to a pseudo-Anosov or the identity homeomorphism.  
\end{itemize}
An element of $\Mod(S)$ is called {\it pure} if the homotopy class contains a pure homeomorphism. Note that we allow $C = \emptyset$ so that pseudo-Anosov homeomorphisms are pure. Recall that for a mapping class $f\in \Mod(S)$ we let $\tau(f)$ denote its translation length in $\teich(S)$.  We can then break pure mapping class elements into three categories along the lines of Bers's classification of surface diffeomorphisms:  if $f\in \Mod(S)$ is pure, then 
\begin{itemize}
\item[(i)] $\tau(f) > 0$ and realized, so that $f$ is a (full) pseudo-Anosov, 
\item[(ii)] $\tau(f)>0$ and not realized, so that $f$ induces a pseudo-Anosov homeomorphism on some component of $S$ cut along the canonical reduction system for $f$ (we will call these {\it partial pseudo-Anosov}), or 
\item[(iii)] $\tau(f)=0$ and not realized, so that $f$ is a Dehn twist about a multicurve, which we will call a {\it multi-twist}.  
\end{itemize}
We will consider homogeneity with respect to cyclic subgroups generated by each type of pure mapping class in turn.

%Partial pseudo-Anosov
\subsection{Full and Partial Pseudo-Anosov  Mapping Classes}

Let $S$ be a closed surface and $f\in \Mod(S)$ be a pure partial pseudo-Anosov mapping class.  Then there exists a multicurve $C$ and a representative of $f$, which we will also call $f$, such that $f$ fixes $C$ pointwise.  Let $R$ be a component of the (possibly disconnected) surface resulting from cutting $S$ along $C$ such that $f|_R$ is pseudo-Anosov.  We can build a punctured surface $F$ by gluing punctured disks to each of the boundary components of $R$, so that $R$ is embedded in $F$.  Furthermore, since $f$ restricted to $\partial R$ is the identity, we can extend $f|_R$ to a map $\hat f: F\to F$ by defining $\hat f|_R = f|_R$ and $\hat f|_{F\smallsetminus R}=id$.  We have constructed $\hat f$ so that $[\hat f]\in \Mod(F)$ is a full pseudo-Anosov map on a punctured surface and our first goal will be to relate the the translation length, $\tau(f)$, of $f$ in $\Teich(S)$ to the translation length, $\tau(\hat f)$, of $\hat f$ in $\Teich(F)$. 

\begin{Lem}\label{lem:partial}
Let $S,f, F, \hat f$ be defined as above, then $\tau(f)\geq \tau(\hat f)$.
\end{Lem}

\begin{proof}
Recall that $\tau(f)$ is not realized, so let $\{(X_n,\vp_n)\}$ be a sequence in $\teich(S)$ and $f_n:X_n\to X_n$ be the Teichm\"uller map in the homotopy class of $\vp_n\circ f\circ \vp_n^{-1}$ so that $\lim_{n\to\infty} K(f_n) = e^{2\tau(f)}$.  Define $R_n$ to be the geometric straightening of $\vp_n(R)$ in $X_n$ so that $\partial R_n$ is a disjoint union of simple closed geodesics. The collaring lemma provides disjoint neighborhoods around each boundary component of $R_n$; let $N_n$ be the union of these neighborhoods. We can then pick points $x_n\in R_n\smallsetminus N_n$ such that $f_n(x_n) \in R_n \smallsetminus N_n$. The sequence - possibly a subsequence - of pointed surfaces $(X_n, x_n)$ converges geometrically to $(X_\infty, x_\infty)$, where $X_\infty$ is homeomorphic to $F$ as the collection of curves permuted by $f$ must be pinched. This convergence is clear as this limit agrees with the visual limit from the viewpoint of $x_n$. With this setup we will construct a quasiconformal map on $X_\infty$ that has the same translation length in $\teich(F)$ as $\hat f$ and smaller dilatation then $\lim_{n\to\infty} K(f_n)$.

We will want to work in the hyperbolic plane; in particular, we will use the disk model $(\bd, d_H)$, where $\bd = \{z\in \bc\colon |z|<1\}$ and $d_H$ is the hyperbolic metric.  Let us identify the universal cover of $(X_n, x_n)$ with $(\bd, 0)$ and let $\Gamma_n< \mathrm{Isom}(\bd)$ such that $X_n = \bd/\Gamma_n$. We may assume that our marking $\vp_n: S\to X_n$ induces the representation $\rho_n = (\vp_n)_*: \pi_1 S\to \Gamma_n$.  We note that the $\Gamma_n$ converge to a group $\Gamma_\infty$ such that $\bh^2/\Gamma_\infty=X_\infty$.  Let $\tilde y_n$ be a lift of $f(x_n)$ such that $d_H(0, \tilde y_n) = d_X(x_n, f(x_n))$, then choose a lift $\tilde f_n: \bd\to \bd$ of $f_n$ with $\tilde f_n(0) = \tilde y_n$. By compactness, the sequence of points $\{\tilde y_n\}$ must have a convergent subsequence, which we also call $\{\tilde y_n\}$, in $\overline{\bd} = \{z\in \bc: |z|\leq 1\}$. Set $\tilde y_\infty  = \lim_{n\to\infty} \{\tilde y_n\}$, then as the $x_n\in R_n$ have been chosen to avoid going up the cusp, we see that $\tilde y_\infty \in \bd$. Let $y_n\in X_n$ be the projection of $\wt y_\infty$ to $X_n$. Define $h_n: X_n\to X_n$ such that $h_n$ is isotopic to the identity, $h_n(f(x_n)) = y_n$ and $\lim_{n\to\infty} K(h_n) = 1$.  Now $g_n=h_n\circ f_n: X_n \to X_n$ with $g_n(x_n) = y_n$; in particular, we can choose lifts $\tilde g_n:\bd\to\bd$ of the $g_n$ with $\tilde g_n(0) = \tilde y_\infty$. 

The family of $K$-quasiconformal maps $$\{g:\bd\to \bd\colon K(g)\leq K \text{ and } g(0) = \tilde y_\infty\}$$ is normal \cite{hubbard}; therefore, the sequence $\{\tilde g_n\}$ of quasiconformal maps has a convergent subsequence, which we also call $\{\tilde g_n\}$.  Define $\tilde g_\infty = \lim_{n\to\infty}\{\tilde g_n\}$, so that $\tilde g_\infty(0) = \tilde y_\infty$ and 
$$K(\tilde g_\infty) = \lim_{n\to\infty} K(\tilde g_n) = \lim_{n\to\infty} K(g_n) \leq \lim_{n\to\infty}[ K(h_n)\cdot K(f_n)] = e^{2\tau(f)}.$$
It is left to show that $\tilde g_\infty$ descends to a map $g_\infty:X_\infty\to X_\infty$ and $\tau(\hat f) \leq \frac12\log K(g_\infty).$

In order to finish the proof we will look at a particular definition of the geometric limit (details for geometric limits can found in $\S E.1$ in \cite{ben}).  Let $p_n: \bh^2\to X_n$ be the canonical projections (where we identify $X_n = \bh^2/\Gamma_n$).  As the sequence $(X_n, x_n)$ converges to $(X_\infty, x_\infty)$ geometrically, we can find bilipschitz maps $\tilde\psi_n: \overline{B(0,r_n)}\to \bh^2$, where $B(z,r)$ is the ball of radius $r$ about $z$, such that $\tilde\psi_n(0)=0$, the $\tilde\psi_n$ converge to the identity on $\bh^2$, and for all $z_1,z_2$ in the domain of $\tilde\psi_n$ 
\begin{equation}\label{eq:equivariance}
p_\infty(z_1)=p_\infty(z_2) \iff p_n(\tilde\psi_n(z_1))=p_n(\tilde\psi_n(z_2)).
\end{equation}
In particular, the maps $\tilde\psi_n^{-1}\circ \tilde g_n\circ \tilde \psi_n$ converge to $\tilde g_\infty$.  Combining $\eqref{eq:equivariance}$ with the fact that $\tilde g_n$ is $\Gamma_n$-equivariant we see that $\tilde\psi_n^{-1}\circ \tilde g_n\circ \tilde \psi_n$ is $\Gamma_\infty$-equivariant on its domain.  This implies that $\tilde g_\infty$ is $\Gamma_\infty$-equivariant and descends to $g_\infty: X_\infty\to X_\infty$. 

It is left to show $\tau(\hat f) \leq \frac12\log K(g_\infty)$.  Condition \eqref{eq:equivariance} implies that the maps $\tilde \psi_n$ descend to $\psi_n: K_n \hra X_n$, where $K_n$ is a compact set in $X_\infty$.  From above we know the domain of $\psi_n^{-1}\circ g_n\circ \psi_n$ is converging to $X_\infty$ and $\psi_n^{-1}\circ g_n\circ \psi_n$ is converging to $g_\infty$.  Choose $N$ such that for $n>N$ if removing the domain of $\psi_n^{-1}\circ g_n\circ \psi_n$ from $X_\infty$ results in a disjoint union of punctured disks.  We can then extend $\psi_n^{-1}\circ g_n\circ \psi_n: X_\infty\to X_\infty$ without affecting convergence.  We therefore see that for large $n$ that $\psi_n^{-1}\circ g_n\circ \psi_n$ is homotopic to $g_\infty$, which implies
$$g_\infty \simeq \psi_n^{-1}\circ\vp_n\circ f\circ \vp_n^{-1}\circ\psi_n.$$
On the domain of interest, we are really looking at restricting the $\vp_n$ and $f$ to $R$ and then extending. In fact, we see that 
$$g_\infty \simeq \psi_n^{-1}\circ\vp_n\circ \hat f\circ \vp_n^{-1}\circ\psi_n.$$
We can think of an extension of $\psi_n^{-1}\circ \vp_n|_R$ as a marking $F\to X_\infty$, which implies $\tau(\hat f) \leq \frac12\log K(g_\infty) \leq \tau(f)$ as desired.
\end{proof}

We will consider both full and partial pseudo-Anosov homeomorphisms at the same time.  We will rely on a result of Penner \cite{penner}, which provides a lower bound for the dilatation of a pseudo-Anosov $f\in \Mod(S)$: 
$$\log \lambda(f) \geq \frac{\log 2}{|\chi(S)|},$$
where $\chi(S)$ denotes the Euler characteristic of $S$.  
This holds for both closed and punctured surfaces.

\begin{Thm}
There exists a constant $K_A>1$ such that if a closed hyperbolic surface $X$ is $\Gamma_K$-homogeneous, where $\Gamma=\la f\ra$ with $f\in\Mod(X)$ either pseudo-Anosov or partial pseudo-Anosov, then $K\geq K_A$. In particular, we have
$K_A \geq 1.42588.$ 
\end{Thm}

\begin{proof}
Let $[f]\in \Mod(X)$ and $R\subseteq X$ a connected subsurface such that $f|_R$ is pseudo-Anosov and $f(R)$ is isotopic to $R$. Note that in the case $f$ is not reducible, then $R=X$.  We will keep with our notation above, so that we can extend $f|_R$ to $\hat f: F\to F$, where $F$ is a punctured surface in the reducible case or again $F = X$ and $\hat f = f$ in the pseudo-Anosov case.  If we let $\tau(\hat f)$ denote the translation length of $\hat f$ in $\teich(F)$, then, as $|\chi(F)|\leq |\chi(X)|$, we have $\tau(\hat f) \geq \frac{\log 2}{12(g-1)}$, where $g$ is the genus of $X$ (see \cite{penner}).
Let $m\in\bz$ such that 
$$\frac{m\log 2}{6(g-1)}\geq \log K.$$
As $\widehat{(f^2)} = \hat f^2$ and $\tau(\hat f^2) = 2\tau(\hat f)$, we find
$$\tau(f^m) \geq \tau(\hat f^m) = m\tau(\hat f) \geq \frac{m\log2}{12(g-1)}\geq\frac12\log K.$$
In particular, $K(f^m)\geq K$. We can now appeal to Lemma \ref{lemma} with $n\leq 2m+1$ (accounting for negative powers and the identity) to find that 
$$K\geq \mu_g(K),$$
where we define 
$$\mu_g(K) = \sqrt{\psi\left(2\arccosh\left(\frac{2\log 2}{12(g-1)\log K +\log2}(g-1) +1\right)\right)}.$$
As $\mu_g(K)$ increases with $g$, we have that $K\geq \mu_2(K)$.   For $K\geq 1$, we see that $\mu_2(K)$ is decreasing and so there exists a unique solution to $K-\mu_2(K) = 0$,  call it $K_{A}$. A computation shows that $K_{A} = 1.42588...$ and the result follows.
\end{proof}

%Dehn Twists
\subsection{Multi-twists}

We start this section with finding a lower bound for the dilatation of a quasiconformal homeomorphism homotopic to a multi-twist. We do this by understanding the map induced on the boundary of the hyperbolic plane. Let $X$ be a closed hyperbolic  surface and $f\in \QC(X)$, then by identifying the universal cover of $X$ with $\bh^2$ we can choose $\wt f: \bh^2\to\bh^2$ to be a lift of $f$. Furthermore, we can extend $\wt f$ to the boundary of $\bh^2$ continuously, which we identify with $\overline \br$.  Let $\overline f: \overline\br\to \overline\br$ be the restriction of $\wt f$ to $\overline\br=\partial \bh^2$.  We can choose $\wt f$ such that $\overline f(\infty) = \infty$. In this setup there exists an $M$ such that $\overline f$ is $\br$-quasisymmetric with modulus $M$, that is
$$\frac1M\leq \frac{\overline{f} (x+t) - \overline f(x)}{\overline f(x) -\overline f(x-t)} \leq M,$$
for all $x\in \br$ and $t>0$ (see $\S4.9$ of \cite{hubbard}).  Sharp bounds are known for the modulus $M$ above associated to a $K$-quasiconformal homeomorphism of $\bh^2$:  define $$\lambda(K) = \frac{1}{(\mu^{-1}(\pi K/2))^2}-1,$$
where $\mu(r)$ is the modulus of the Gr\"otsch ring whose complementary components are $\overline{\mathbb{B}^2}$ and $[1/r,\infty]$ for $0<r<1$.  Then (see \cite{lehto}) we have
\begin{equation}\label{eq:qs}
\frac1{\lambda(K(f))}\leq \frac{\overline{f} (x+t) - \overline f(x)}{\overline f(x) -\overline f(x-t)} \leq \lambda(K(f)).
\end{equation}
If $f$ is homotopic to a multi-twist, then this is enough information to produce a lower bound for $K(f)$ in terms of the lengths of the curves $f$ twists about.

\begin{Lem}\label{twist}
Let $X$ be a closed hyperbolic surface and $f\in \QC(X)$  be homotopic to a multi-twist $T_C$ about a multicurve $C = \{\gamma_1, \ldots, \gamma_n\}$, so that $T_C = T_{\gamma_1}^{m_1}\circ\cdots \circ T_{\gamma_n}^{m_n}$. If $m = |m_k|, \ell=\ell_X(\gamma_k)$ such that $m\ell = \max_i\{|m_i|\cdot\ell_X(\gamma_i)\}$, then 
$$ K(f) \geq \frac 2\pi \, \mu\left(\sqrt{\frac2{2+e^{(m-1)\ell}+e^{(m-\frac12)\ell}}}\right),$$
where $\mu(r)$ is the modulus of the Gr\"otsch ring whose complementary components are $\overline{\mathbb{B}^2}$ and $[1/r,\infty]$ for $0<r<1$.
\end{Lem}

\begin{figure}[t]
\includegraphics[scale=.5, clip=false]{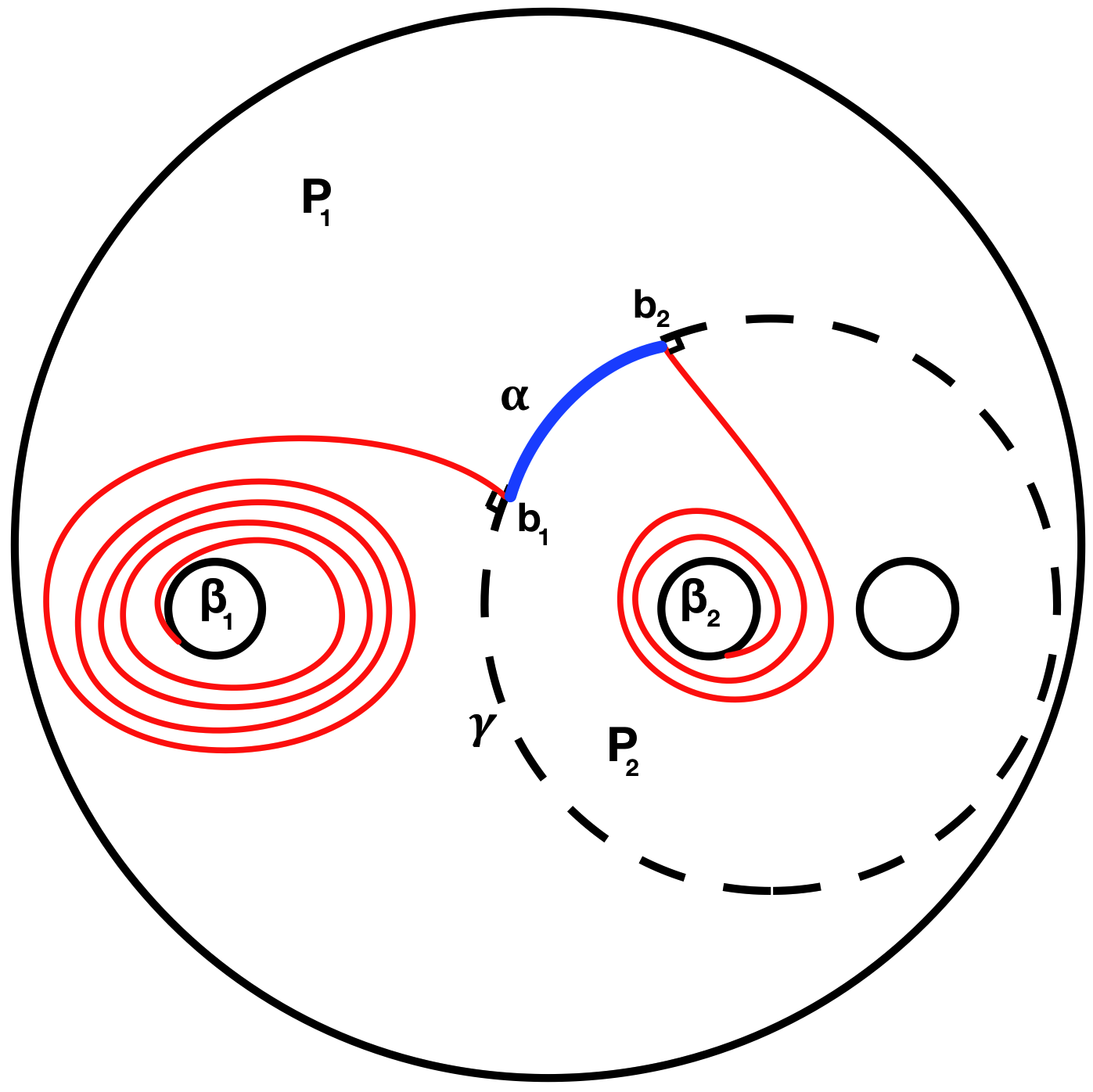}
\caption{A 4-punctured sphere in $X$ with $\gamma$ bounding two embedded pairs of pants.  The curve $\al$ intersects $\gamma$ once and spirals towards both $\be_1$ and $\be_2$ so that it is disjoint from all boundary components.}\label{fig:4sphere}
\end{figure}

\begin{proof}
Let $\gamma=\gamma_k$ so that $\ell = \ell_\sx(\gamma)$ and extend the collection $C = \{\gamma_1, \ldots, \gamma_n\}$ of disjoint simple closed curves to a maximal collection, call it $C'$, giving a pants decomposition for $X$.  We want to construct an infinite simple complete geodesic in $X$, which does not intersect any element of $C'$ other than $\gamma$. First assume that $\gamma$ bounds two pairs of pants, $P_1$ and $P_2$  as in Figure \ref{fig:4sphere}.  Let $\beta_i$ be a component of $\partial P_i$  for $i=1,2$ such that $\be_i \neq \gamma$, then there exists a geodesic ray in $P_i$ spiraling towards $\be_i$ and meeting $\gamma$ perpendicularly at $b_i$. In $X$, $P_1$ and $P_2$ are glued together with a twist along $\gamma$, so we can create a geodesic $\al$ by connecting the two rays via an arc on $\gamma$ connecting the images of $b_1$ and $b_2$ in $X$ and pulling this curve tight. The other possibility is that $\gamma$ bounds a single pair of pants $P$.  In $P$ we have two copies of $\gamma$ and one other boundary component.  There exists a ray emanating perpendicularly from each copy of $\gamma$ spiraling towards this other component such that these two rays are disjoint.  We then construct $\alpha$ from these rays as above.  We see that $\alpha$ is our desired complete geodesic.

\begin{figure}[t]
\includegraphics[scale=0.5, clip=false]{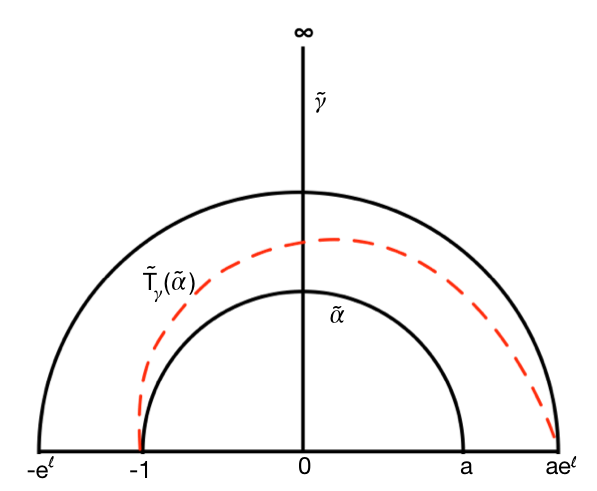}
\caption{Lifts of $\al$ and $\gamma$ in the upper half plane.  Also drawn is a copy of $\tilde\al$ under a translation by the element of $\pi_1X$ representing $\gamma$.  The dotted geodesic is the image of $\tilde \alpha$ under the lift of a Dehn twist about $\gamma$.}\label{fig:uhp}
\end{figure}

We can identify the universal cover $\wt X$ of $X$ with the upper half plane $\{z\in\bc\colon \mathrm{Im}(z)>0\}$ so that we have lifts $\wt\gamma, \wt\alpha$ of $\gamma, \alpha$ in the configuration showed in Figure \ref{fig:uhp}.  Let $T_\gamma:X\to X$ be a left Dehn twist about $\gamma$ and let $\wt T_\gamma:\bh^2\to\bh^2$ be a lift of $T_\gamma$ fixing $\wt\gamma$. Let $[x,y]$ denote the geodesic in $\bh^2$ with endpoints $x,y\in \partial \bh^2$. In our setup, $\wt\alpha = [-1,a]$ and we see that $\wt T_\gamma(\wt \alpha)$ is homotopic to the dotted curve shown in Figure \ref{fig:uhp} and has endpoints $[-1,ae^\ell]$. By iterating this map, we can construct a family of geodesics $\{\alpha_n\}$ in $X$ that are the projection of $\wt T_\gamma^n(\wt \alpha) = [-1, ae^{n\ell}]$.  Furthermore, every $\alpha_n$ is an infinite simple complete geodesic in $X$ that does not intersect any element of $C'$ other than $\gamma$.  We can then find an integer $k$ such that $ae^{k\ell}\in[\frac12(e^{-\ell}+e^{-\ell/2}), \frac12(1+e^{\ell/2})]$; define $\wt \beta = [-1,ae^{k\ell}] = [-1,b]$ so that the image of $\wt\beta$ is $\beta=\alpha_k$. 

We now want to investigate $K=K(f)$ by studying $\overline f: \partial \bh^2\to \partial\bh^2$, which is the induced boundary map from the lift $\wt f: \bh^2\to\bh^2$ fixing $0,-1,\infty$.  As two homotopic maps induce the same boundary map on $\bh^2$, we have $\bar f = \overline T_C$ (it is convenient to think of $\bar f$ as the map on $\partial \bh^2$ coming from a left earthquake along the complete lift of the multicurve $C$, see \cite{k} for the definition of an earthquake).  Let us assume for now that $\frac12(e^{-\ell}+e^{-\ell/2})\leq b \leq 1$ and that $f$  twists left about $\gamma$ (if not we can just study $f^{-1}$).  By construction $\be$ is infinite in $X$, $\be$ intersects $\gamma$ exactly once, and $\be\cap \gamma_i = \emptyset$ for $i\neq k$; this implies that $\wt\gamma$ is the only geodesic in the full lift of $C$ that $\wt\be$ intersects. Therefore, we know that $\overline f(b) = be^{m\ell}$ and also that $[-1,b^{m\ell}]$ and $[-1, \overline f(1)]$ do not intersect as $[-1,b]$ and $[-1,1]$ do not. In particular, we must have that $\overline f(1) \geq b^{m\ell}$.  This yields:
$$\lambda(K) \geq \frac{\overline f(1) - \overline f(0)}{\overline f(0)-\overline f(-1)} = \overline f(1) \geq  be^{m\ell}\geq \frac12\left(e^{(m-1)\ell}+e^{(m-\frac12)\ell}\right).$$
From above we can write 
$$ K = \frac2\pi \mu\left(\sqrt{\frac{1}{\lambda(K)+1}}\right),$$
and as $\mu$ is a decreasing function (see \cite{lehto}), we have
$$ K \geq \frac 2\pi \, \mu\left(\sqrt{\frac2{2+e^{(m-1)\ell}+e^{(m-\frac12)\ell}}}\right).$$

Now assume that $1\leq b \leq \frac12(1+e^\ell)$.  Furthermore since $K(f) = K(f^{-1})$ for any quasiconformal map, we may assume that $f$ twists to the right along $\gamma$.  We have the same exact setup as before, except this time the inequality is as follows:
$$\frac1{\lambda(K)} \leq \frac{\overline f(1) - \overline f(0)}{\overline f(0)-\overline f(-1)} = \overline f(1) \leq \overline f(b) = be^{-m\ell} \leq \frac12\left(e^{-m\ell}+e^{(\frac12-m)\ell}\right) ,$$
yielding
$$K \geq \frac2\pi\, \mu \left(\sqrt{\frac{1+e^{\frac \ell2}}{1+e^{\frac \ell2}+2e^{m\ell}}}\right).$$
As $\mu$ is decreasing, for $\ell\geq 0$ the first inequality for $K$ is always smaller.
\end{proof}

We saw in Theorem \ref{diameter} that a hyperbolic surface $X$ with a short curve has a large homogeneity constant.  We leverage this with the above lemma to get a universal bound for the homogeneity constant with respect to a subgroup of $\Mod(X)$ generated by a multi-twist.  

\begin{Thm}
There exists a constant $K_D>1$ such that if a closed hyperbolic  surface $X$ is $\Gamma_K$-homogeneous, where $\Gamma=\la f\ra<\Mod(X)$ with $f$ being a muti-twist, then $K\geq K_D$. In particular, we have
$K_D \geq1.09297.$ 
\end{Thm}

\begin{proof}
Let $\ell = \ell(X)$ be the systole of $X$.  From the definition of $m(2,K)$ in Theorem \ref{diameter} given in \cite{bcmt} and the inequality $\ell\geq m(2,K)$, we have 
\begin{equation}\label{eq:phi}
K\geq \frac{\log\left(\frac12\tanh\frac{\ell}2\right)-\log{2e}}{\log\left(\frac12\tanh \frac{d_2}2\right)-\log{2e}} \equiv \Phi(\ell),
\end{equation}
where $d_2$ is defined such that every closed hyperbolic surface contains an embedded hyperbolic disk of diameter $d_2$.  It is shown in \cite{yamada} that we can take $d_2 = 2\log(1+\sqrt{2})$.
From Lemma \ref{twist} we have
\begin{equation}\label{eq:psi}
K(f) \geq \frac 2\pi \, \mu\left(\sqrt{\frac2{3+e^{\frac12\ell}}}\right) \equiv \Psi(\ell).
\end{equation}
Now, $\Phi$ is decreasing on $\br^{>0}$ with $\Phi(0) = +\infty$ and $\Psi$ is an increasing function on $\br^{>0}$ with $\Psi(0)=1$; hence, there exists a unique value $L$ such that $\Phi(L)=\Psi(L)$.  We note that $L \approx 1.33994$ and $\Phi(L) \approx 1.09297$.  If $\ell \leq L$, then $K\geq \Phi(L)$.  Assume $\ell\geq L$ and $K<\Psi(L)$. Then $K(f) \geq \Psi(L)$ and every element in $\Gamma_K$ is isotopic to the identity: this case is handled in \cite{bbc} and tells us it must be that $K\geq 1.626> \Psi(L)$.  This contradiction proves the theorem.
\end{proof}

Theorem \ref{cyclic} is now just a corollary of the previous two sections with setting $K_C =\min\{K_D,K_A\}$. 

%Torsion-free Subgroups

\section{Torsion-Free Subgroups}
In this section we investigate a lower bound for the homogeneity constant of a surface in terms of its genus.  The idea is to find a lower bound for the dilatation of a quasiconformal map on a thick surface.  Periodic elements create serious difficulties that we do not know how to deal with, so we will restrict ourselves to the torsion-free case.

\bigskip\noindent
{\bf Theorem \ref{thm:torsion}}  {\it
Let $X$ be a closed hyperbolic surface and suppose $\Gamma<\Mod(X)$ is torsion-free.  If $X$ is $\Gamma_K$-homogeneous, then 
$$\log K \geq \frac{1}{7000g^2},$$
where $g$ is the genus of $X$.}

\begin{proof}
Let $\sF = \{f \in \Gamma\colon \log K(f) < 7000^{-1}g^{-2}\}$, then our goal will be to show that $\sF = \{id\}$.  The first observation is that $\sF$ cannot contain any pseudo-Anosov or pure partial pseudo-Anosov elements.  This is seen by combining the bounds in \cite{penner} already mentioned and Lemma \ref{lem:partial}. 

We can find $\ell_0$ such that $\log \Phi(\ell_0) >1$, where $\Phi$ is defined in \eqref{eq:phi}; in particular, we can take $\ell_0 = 1.8$.  Furthermore, since we know that if $\ell(X) < \ell_0$ then $K>\Phi(\ell_0)>\exp(g^{-2})$.  Therefore, we may assume $\ell(X)>\ell_0$ and so $\sF$ cannot contain any multi-twists as any mutli-twist will have dilatation bigger than $\Psi(\ell_0) =1.12 $, where $\Psi$ is defined in \eqref{eq:psi}.  We are left with mapping classes of the form $f$ where some power of $f$ is either a partial pseudo-Anosov or multi-twist.  

Let us first consider the partial pseudo-Anosov case: we can find a subsurface $R\subset X$ and a $k>0$ such that $f^k$ fixes the isotopy class of $R$ and $f^k|R$ is pseudo-Anosov. There are at most $\chi(X)/\chi(R)$ copies of $R$ permuted by $f$ in $X$, therefore we may choose $k\leq \chi(X)/\chi(R)$.  We then have 
$$\log K(f^k)\geq \frac{\log 4}{|\chi(R)|}.$$
It follows that
$${k\cdot|\chi(R)|}\cdot\log K(f) \geq |\chi(R)| \cdot\log K(f^{k}) \geq \log 4,$$
and
$$\log K(f) \geq \frac{\log 4}{k\cdot |\chi(R)|} \geq \frac{\log 4}{|\chi(X)|}.$$
This shows that $f\notin\sF$.

We may now suppose that some power of $f$ is a multi-twist.  Recall $\ell(X)>\ell_0$.  Choose a simple closed curve $\gamma$ and $k>0$ such that $f^k([\gamma])=[\gamma]$.  Define $R_1$ and $R_2$ to be the subsurfaces on either side of $\gamma$ (possibly $R_1=R_2$) such that there exists $n>0$ with $f^n|_{R_1}=f^n|_{R_2} = id$.  Let $R=R_1\cup R_2$, then we can choose $k<\chi(X)/\chi(R)$; furthermore, $f^{2k}$ fixes the isotopy classes of both $R_1,R_2$.  Now choose $m_i$ such that $f^{2km_i}|_{R_i} = id$ for $i=1,2$.  By doubling $R_i$, we see that 
$$m_i\leq 4|\chi(R_i)|+6\leq10|\chi(R_i)|$$
 (recall that for a periodic element $h\in \Mod(S_g)$ that $|\la h\ra|\leq 4g+2 = 2 |\chi(S_g)|+6)$.  This implies $2km_1m_2 < 800g^2$.  The same line of argument as above tells us that
$$2\cdot k\cdot m_1\cdot m_2 \cdot \log K(f) \geq \log\Phi(\ell_0)$$
and
$$\log K(f) \geq \frac{\log \Phi(\ell_0)}{800 g^2}>\frac{1}{7000g^2}.$$
Again we see $f\notin \sF$.

We have exhausted all the torsion free elements in $\Mod(S_g)$; hence, $\sF = \{id\}$ as claimed.  If $\log K < 7000^{-1}g^{-2}$, we can proceed by contradiction as we did in the cyclic multi-twist case:  we must have that the elements in $\Gamma_K$ are isotopic to the identity: this case is handled in \cite{bbc} and implies $K\geq 1.626$, which is larger than our assumption.  This is a contradiction, so we see $\log K>7000^{-1}g^{-2}$.   
\end{proof}

%Continuous Function on Teichmuller and Moduli Space

\section{Functions on Teichm\"uller Space and Moduli Space}

This section looks at building functions on Teichm\"uller space out of measuring the homogeneity constant at a given point.  The statements and techniques follow the related results in \cite{bbc}. For the entirety of this section, let $S$ be a closed orientable surface with $\chi(S) < 0$. Let $\sx = [(X,\vp)]\in \teich(S)$, then given $\Gamma<\Mod(S)$ define 
$$\Gamma_\vp = \{[\vp\circ f\circ \vp^{-1}]\colon f\in\mathrm{Homeo}^+(S) \text{ and } [f]\in \Gamma\} < \Mod(X).$$ We then define $K_\Gamma: \teich(S) \to (1,\infty)$ by 
$$K_\Gamma([(X,\vp)]) = \min\{K\colon X \text { is } (\Gamma_\vp)_K\text{-homogeneous}\}.$$

\begin{Lem}
Given $\Gamma<\Mod(S)$, the function $K_\Gamma:\teich(S) \to (1,\infty)$ exists and is well-defined.
\end{Lem}

\begin{proof}
We first need to prove that $K_\Gamma$ exists, i.e. that the minimum exists.  Let $X$ be a hyperbolic  surface and let $\vp:S\to X$ be a diffeomorphism.  Set 
$$K = \inf \{Q\colon X \text { is } (\Gamma_\vp)_Q\text{-homogeneous}\}.$$
We can then find a sequence $\{K_j\}$ converging to $K$ such that $X$ is $(\Gamma_\vp)_{K_j}$-homogeneous. We want to show that $X$ is $(\Gamma_\vp)_K$-homogeneous. 

Let $x,y\in X$, then we can find a $K_j$-quasiconformal homeomorphisms $f_j$ such that $f_j(x) = y$.  Pick lifts $\ti x, \ti y\in \bh^2$  and $\ti f_j:\bh^2\to\bh^2$ of $x,y,$ and $f_j$, respectively, such that $\tilde f_j(\tilde x) = \ti y$. We recall that the family of all $Q$-quasiconformal homeomorphisms of $\bh^2$ sending $\ti x$ to $\ti y$ is normal (see corollary 4.4.3 in \cite{hubbard}).  Therefore, there exists a subsequence of $\{\ti f_j\}$ that converges to a $K$-quasiconformal homeomorphism $\ti f: \bh^2\to\bh^2$ with $\ti f(\ti x)=\ti y$.  Furthermore, $\ti f$ descends to a $K$-quasiconformal mapping $f:X\to X$.  It is left to show that $[f]\in \Gamma_\vp$. As the connected components of $\QC(X)$ are given by isotopy classes, we must have that for $j$ large $[f_j]=[f]$ and as each $[f_j]\in \Gamma_\vp$, so is $[f]$. This shows the minimum exists.

As a point in Teichm\"uller space is an equivalence class we must check that $K_\Gamma$ is well-defined.  Let $(X,\vp)=(X,\psi)\in \teich(S)$, so that $\vp$ and $\psi$ are isotopic.  As $\Mod(X)$ is defined up to isotopy, it is clear that $\Gamma_\vp = \Gamma_\psi$ and $K_\Gamma((X,\vp))=K_\Gamma((X,\psi))$.  Now let $(X,\vp) = (Y,\xi) \in \teich(S)$, so that $\vp\circ\xi^{-1}\simeq I$ for some conformal map $I:Y\to X$. As conformal maps preserve quasiconformal dilatations it is clear $K_\Gamma((Y,\xi))=K_\Gamma((X, I\circ \xi))$.  By definition $I\circ\xi \simeq \psi$, so that by the previous argument $K_\Gamma((X,\vp)) = K_\Gamma((Y,\xi))$.  This shows that $K_\Gamma:\teich(S)\to (1,\infty)$ is well-defined.  
\end{proof}

\noindent We now associated to each subgroup of the mapping class group a continuous function of Teichm\"uller space.  In the following we closely adhere to the proof of Lemma 7.1 in \cite{bbc}.

\begin{Prop}
For $\Gamma<\Mod(S)$, the function $K_\Gamma:\teich(S) \to (1,\infty)$ is continuous.
\end{Prop}

\begin{proof}
We will prove continuity in two steps: we will first prove that $K_\Gamma$ is lower semicontinuous and then that it is upper semicontinuous.  We make the following definitions for the entirety of the proof:
Let $\{\sx_n\} = \{(X_n,\vp_n)\}$ be a sequence in $\teich(S)$ converging to $\sx = (X,\vp)\in \teich(S)$. Let $f_n = \vp\circ\vp_n^{-1}: X_n\to X$ and observe $\lim_{n\to\infty} K(f_n) = 1$. 

Pick $x,y\in X$ and set $x_n = f_n^{-1}(x)$ and $y_n=f_n^{-1}(y)$.  Then there is a $K_\Gamma(X_n)$-qc mapping $g_n:X_n\to X_n$ such that $g_n(x_n)=y_n$ with $[g_n]\in \Gamma_{\vp_n}$.  Let $\{\sx_{n_j}\}$ be a subsequence of $\{\sx_n\}$ such that $\lim K_\Gamma(\sx_{n_j}) = \lim\inf K_\Gamma(\sx_n)$. As $f_{n_j}\circ g_{n_j}\circ f_{n_j}^{-1}: X\to X$ with $f_{n_j}\circ g_{n_j}\circ f_{n_j}^{-1}(x)=y$ and $\lim K(f_{n_j}\circ g_{n_j}\circ f_{n_j}^{-1}) \leq \liminf K(f_{n_j})^2K(g_{n_j})= \liminf K(g_{n_j})$ we can pass to another subsequence, still labelled $\{\sx_{n_j}\}$, such that $f_{n_j}\circ g_{n_j}\circ f_{n_j}^{-1}$ converges to a quasiconformal mapping $g:X\to X$  such that $g(x)=y$ (this is again due to normality as in the above lemma). For $j$ large we must have that $f_{n_j}\circ g_{n_j}\circ f_{n_j}^{-1}$ is homotopic to $g$, again as the connected components of $\QC(X)$ are given by isotopy classes.  As $g_{n_j}\in \Gamma_{\vp_{n_j}}$ we have $[g]\in \Gamma_{f_{n_j}\circ\vp_{n_j}}$, but $f_{n_j}\circ\vp_{n_j} = \vp$, so that $[g] \in \Gamma_\vp$. By our setup we now have
$$K(g) \leq\lim\inf K(g_{n_j}) \leq \lim K_\Gamma(\sx_n) = \liminf(\sx_n).$$ 
As $x,y$ were arbitrary 
$$K_\Gamma(\sx) \leq \liminf K_\Gamma(\sx_n).$$
Therefore, $K_\Gamma$ is lower semicontinuous. 

It is left to show that $K_\Gamma$ is upper semicontinuous.  Fix $n$ and choose $x_n,y_n \in X_n$ and set $x = f_n(x_n)$ and $y = f_n(y_n)$.  Then there exists a $K_\Gamma(\sx)$-qc mapping $g_n: X\to X$ such that $g_n(x) = y$.  We then have that $h_n = f_n^{-1}\circ g_n\circ f_n$ is a qc mapping of $X_n$ such that $h_n(x_n) = y_n$ and $[h_n]\in \Gamma_{\vp_n}$. Furthermore, 
$$K(h_n) \leq K(f_n)^2K(g_n) \leq K(f_n)^2K_\Gamma(\sx).$$
As $x_n, y_n$ were arbitrary we have that 
$$K_\Gamma(\sx_n) \leq K(f_n)^2K_\Gamma(\sx)$$
and thus
$$\lim\sup K_\Gamma(\sx_n) \leq \lim K(f_n)^2K_\Gamma(\sx) = K_\Gamma(\sx).$$
Therefore, $K_\Gamma$ is upper semicontinuous.
\end{proof}

It is natural to ask when these functions descend to functions on Moduli space.  Recall that if $X\in \cM(S)$, then two points $\sx,\sy \in \teich(S)$ are in the preimage of $X$ under the projection $\teich(S) \to \cM(S)$ if there exists $[f]\in \Mod(S)$ with $\sy = [f]\cdot \sx$.  If $\sx = [(X,\vp)]$, then $[f]\cdot \sx = [(X,\psi)]$ with $\psi = \vp\circ f^{-1}$.  Given a normal subgroup $\Gamma \triangleleft \Mod(S)$, then by definition we have
\begin{align*}
\Gamma_\psi &= \{[\psi\circ g\circ \psi^{-1}]\colon g\in\mathrm{Homeo}^+(S) \text{ and } [g]\in \Gamma\}\\
	& = \{[\vp\circ f^{-1}\circ g\circ f\circ \psi^{-1}]\colon g\in\mathrm{Homeo}^+(S) \text{ and } [g]\in \Gamma\}\\
	& = \{[\vp\circ g'\circ \vp^{-1}]\colon g'\in\mathrm{Homeo}^+(S) \text{ and } [g']\in \Gamma\} \\
	& = \Gamma_\vp
\end{align*}
As $\Gamma_\psi = \Gamma_\vp$ it is clear that $K_\Gamma(\sx) = K_\Gamma(f\cdot \sx)$. This proves the following:

\begin{Prop}
For a normal subgroup $\Gamma\triangleleft \Mod(S)$, the function $K_\Gamma:\teich(S)\to (1,\infty)$ descends to a continuous function $K_\Gamma:\cM(S)\to (1,\infty)$.
\end{Prop}

\begin{Rem}
The normality of the subgroup in the above lemma is required: Dehn twists about curves with different lengths have different dilatations and all Dehn twists about non-separating simple closed curves are conjugates.  If we take $\Gamma = \la f\ra$ where $f\in \Mod(S)$ is a Dehn twist about a curve $\gamma$, then for $X\in \cM(S)$ with $\ell(X)$ very small we can choose $\vp:S\to X$ and $\psi:S\to X$ and some $K$ such that $|(\Gamma_\vp)_K| =1$ (where $\ell(\vp(\gamma))$ is very large) and $|(\Gamma_\psi)_K| = 1000$ (where $\ell(\psi(\gamma))$ is very small). In the latter case you have more quasiconformal maps at your disposal.
\end{Rem}

\footnotesize 
\nocite{*}
\bibliographystyle{amsalpha}	
\bibliography{qch}

\providecommand{\bysame}{\leavevmode\hbox to3em{\hrulefill}\thinspace}
\providecommand{\MR}{\relax\ifhmode\unskip\space\fi MR }
\providecommand{\MRhref}[2]{%
  \href{http://www.ams.org/mathscinet-getitem?mr=#1}{#2}
}
\providecommand{\href}[2]{#2}
\begin{thebibliography}{BTMRT11}

\bibitem[Ber78]{bers}
Lipman Bers, \emph{An extremal problem for quasiconformal mappings and a
  theorem by {T}hurston}, Acta Mathematica \textbf{141} (1978), no.~1, 73--98.

\bibitem[BP92]{ben}
Riccardo Benedetti and Carlo Petronio, \emph{Lectures on hyperbolic geometry},
  Springer Verlag, 1992.

\bibitem[BTBCT07]{bbc}
Petra Bonfert-Taylor, Martin Bridgeman, Richard~D Canary, and Edward~C Taylor,
  \emph{Quasiconformal homogeneity of hyperbolic surfaces with fixed-point full
  automorphisms}, Mathematical Proceedings of the Cambridge Philosophical
  Society, vol. 143, Cambridge Univ Press, 2007, pp.~71--84.

\bibitem[BTCMT05]{bcmt}
Petra Bonfert-Taylor, Richard~D Canary, Gaven Martin, and Edward Taylor,
  \emph{Quasiconformal homogeneity of hyperbolic manifolds}, Mathematische
  Annalen \textbf{331} (2005), no.~2, 281--295.

\bibitem[BTMRT11]{btgr}
Petra Bonfert-Taylor, Gaven Martin, Alan~W Reid, and Edward~C Taylor,
  \emph{Teichm{\"u}ller mappings, quasiconformal homogeneity, and non-amenable
  covers of riemann surfaces}, Pure and Applied Mathematics Quarterly
  \textbf{7} (2011), no.~2.

\bibitem[FLM08]{flm}
Benson Farb, Christopher~J Leininger, and Dan Margalit, \emph{The lower central
  series and pseudo-anosov dilatations}, American journal of mathematics
  \textbf{130} (2008), no.~3, 799--827.

\bibitem[FLP79]{flp}
Albert Fathi, Fran{\c{c}}ois Laudenbach, and Valentin Po{\'e}naru,
  \emph{Travaux de thurston sur les surfaces, volume 66 of ast{\'e}risque.
  soci{\'e}t{\'e} math{\'e}matique de france, paris, 1979. s{\'e}minaire
  orsay}, With an English summary \textbf{332} (1979).

\bibitem[FM11]{fm}
Benson Farb and Dan Margalit, \emph{A primer on mapping class groups}, vol.~49,
  Princeton University Press, 2011.

\bibitem[GP76]{gp}
Fredrick~W Gehring and Bruce~P Palka, \emph{Quasiconformally homogeneous
  domains}, Journal d'analyse Mathematique \textbf{30} (1976), no.~1, 172--199.

\bibitem[Gre13]{greenfield}
Mark Greenfield, \emph{A lower bound for {T}orelli-{K}-quasiconformal
  homogeneity}, Geometriae Dedicata (2013), 1--10.

\bibitem[Gro83]{gromov}
Mikhael Gromov, \emph{Filling riemannian manifolds}, J. Differential Geom
  \textbf{18} (1983), no.~1, 1--147.

\bibitem[Hub06]{hubbard}
John~H Hubbard, \emph{Teichm{\"u}ller theory and applications to geometry,
  topology, and dynamics}, vol.~1, Matrix Pr, 2006.

\bibitem[Iva92]{ivanov}
Nikolai~V Ivanov, \emph{Subgroups of teichmuller modular groups}, vol. 115,
  American Mathematical Society, 1992.

\bibitem[Kap01]{kapovich}
Michael Kapovich, \emph{Hyperbolic manifolds and discrete groups}, vol. 183,
  Birkhauser, 2001.

\bibitem[Ker83]{k}
Steven~P Kerckhoff, \emph{The nielsen realization problem}, Annals of
  mathematics \textbf{117} (1983), no.~2, 235--265.

\bibitem[KM11]{markovic}
Ferry Kwakkel and Vladimir Markovic, \emph{Quasiconformal homogeneity of genus
  zero surfaces}, Journal d'Analyse Math{\'e}matique \textbf{113} (2011),
  no.~1, 173--195.

\bibitem[LV73]{lehto}
Olli Lehto and Kalle~I Virtanen, \emph{Quasiconformal mappings in the plane},
  vol. 126, Springer New York, 1973.

\bibitem[Pen91]{penner}
Robert~C Penner, \emph{Bounds on least dilatations}, Proc. Amer. Math. Soc,
  vol. 113, 1991, pp.~443--450.

\bibitem[Tei44]{teich}
Oswald Teichm{\"u}ller, \emph{Ein verschiebungssatz der quasikonformen
  abbildung}, Deutsche Math \textbf{7} (1944), no.~336-343, 8.

\bibitem[Thu79]{thurston}
William~P. Thurston, \emph{The geometry and topology of three-manifolds},
  Princeton University Princeton, 1979.

\bibitem[Thu88]{thurston2}
William~P Thurston, \emph{On the geometry and dynamics of diffeomorphisms of
  surfaces}, Bulletin (New Series) of the American Mathematical Society
  \textbf{19} (1988), no.~2, 417--431.

\bibitem[Vuo88]{vuorinen}
Matti Vuorinen, \emph{Conformal geometry and quasiregular mappings}, Lecture
  Notes in Mathematics, no. 1319, Springer-Verlag, Berlin, 1988.

\bibitem[Wim]{wiman}
A~Wiman, \emph{Ueber die hyperelliptischen curven und diejenigen vom
  geschlechte p= 3}, Welche eindeutigen transformationen in Sich Zulassen and
  Ueber die algebraischen curven von den Geschlechtern p \textbf{4}, no.~5,
  1895--96.

\bibitem[Yam81]{yamada}
Akira Yamada, \emph{On {M}arden's universal constant of fuchsian groups}, Kodai
  Mathematical Journal \textbf{4} (1981), no.~2, 266--277.

\end{thebibliography}

\end{document}